\documentclass[11pt, a4paper, twoside]{article}

\usepackage{t1enc}
\usepackage[latin1]{inputenc}
\usepackage{graphicx}
\usepackage{amsthm}
\usepackage[T1]{fontenc}
\usepackage{amsmath}
\usepackage{amssymb}
\usepackage{stmaryrd}
\usepackage{color}
\usepackage{fourier}
\usepackage{makeidx}
\usepackage{scrtime}
\usepackage{fancyhdr}
\usepackage{enumerate}
 \usepackage[all]{xypic}
\usepackage[mathscr]{eucal}
\usepackage[german,UKenglish]{babel}
\usepackage{longtable}
\usepackage[small,bf]{caption}
\usepackage[backgroundcolor=green!30]{todonotes}

\usepackage{tikz}

\makeatletter
\renewcommand{\big}{\bBigg@{1}}
\renewcommand{\Big}{\bBigg@{1.5}}
\renewcommand{\bigg}{\bBigg@{2.4}}
\renewcommand{\Bigg}{\bBigg@{3.2}}
\newcommand{\bigo}{\bBigg@{1.2}}
\makeatother

\DeclareMathOperator{\data}{data}

 \theoremstyle{plain}
\newtheorem{satz}{Theorem}[section]
\newtheorem{lemma}[satz]{Lemma}

\newtheorem{korr}[satz]{Corollary}

\theoremstyle{definition}
\newtheorem{defi}[satz]{Definition}
 \theoremstyle{remark}

\numberwithin{equation}{section}

\definecolor{sgreen}{rgb}{0.1,0.5,0.0}
\definecolor{sred}{rgb}{0.8,0.0,0.1}
\definecolor{sblue}{rgb}{0.1,0.0,0.4}
\definecolor{green}{rgb}{0.1,0.6,0.1}
\definecolor{lgreen}{rgb}{0.4,0.8,0.4}
\definecolor{red}{rgb}{0.6,0.1,0.1}
\definecolor{lred}{rgb}{0.8,0.4,0.4}
\definecolor{blue}{rgb}{0.1,0.2,0.6}
\definecolor{lblue}{rgb}{0.4,0.2,0.8}

\newcommand{\C}{\mathbb{C}}

\def\bigset(#1|#2){\big\{#1\, \big|\, #2\big\}}
\def\biggset(#1|#2){\bigg\{#1\, \bigg|\, #2\bigg\}}
\def\Bigset(#1|#2){\Big\{#1\, \Big|\, #2\Big\}}
\def\vectdreik(#1;#2;#3){\text{\tiny $\begin{pmatrix} #1 \\ #2 \\ #3 \end{pmatrix}$}}
\def\vectdrei(#1;#2;#3){\begin{pmatrix} #1 \\ #2 \\ #3 \end{pmatrix}}

\newcommand{\dirac}{\mathcal{D}}

\newcommand{\ci}{\mathtt{i}}

\newcommand{\Spinc}{\text{Spin}^{\C}}

\def\ddxi,#1{\frac{\partial}{\partial x_{#1}}}

\newcommand{\R}{{\mathbb{R}}}
\newcommand{\Z}{{\mathbb{Z}}}

\newcommand{\forget}[1]{{\huge$\mathbf{\cdots}$}}

\newcommand{\ddt}{\frac{\partial}{\partial t}}

\def\vekthq(#1,#2,#3,#4){\left(\begin{array}{c} #1 \\ #2 \\ #3 \\ #4 \end{array}\right)}

\newcommand{\Mp}{M^+_{\text{cut}}}
\newcommand{\Mm}{M^-_{\text{cut}}}
\newcommand{\Mpm}{M^\pm_{\text{cut}}}

\newcommand{\lies}{\mathfrak{s}_1}

\newcommand{\up}{\nearrow}

\usepackage{environ}
\NewEnviron{rr}{%
\begin{align*}
  \BODY
\end{align*}
}

\title{$\Spinc$ quantization in odd dimensions}
\author{J. Fabian Meier}
\begin{document}
\maketitle














\tableofcontents

\begin{abstract}
We define and discuss an extension of the $\Spinc$ quantization concept to odd-dimensional manifolds. After that  describe its relation to (the usual) even-dimensional $\Spinc$ quantization and how its famous properties like ``Quantization commutes with reduction'' can be regained in odd dimensions. At the end, we analyze the situation on 3-manifolds and give some examples.


\end{abstract}

\section{Odd-dimensional $\Spinc$ quantization}
\label{sec:odd-dimens-spinc}

Unlike geometric quantization efforts using symplectic or almost complex structures, $\Spinc$ quantization, as it is discussed in \cite{thesisfuchs} and \cite{da2000quantization}, seems to be not so much dependent on the assumption $\dim M = 2m$. Of course, if you naively try to replace $2m$ by $2m+1$ in all definitions, you fail because your equivariant $\Spinc$ Dirac operator will not split, so you cannot define a (non-zero) index.

Nevertheless, if you instead consider a family of operators parametrised by $S^1$ you get an equivariant spectral flow, i.e. an element of $K^1_{S^1}(S^1) \cong K^1(S^1) \otimes R(S^1)$. Through the isomorphism $K^1(S^1) \cong K^0(\{\text{pt}\})$ this can be seen as quantization in the sense of \cite{thesisfuchs}. We give a more precise definition in the next section.

This index is closely related to the even-dimensional index, using maps of the form $M \mapsto M \times S^1$.

\section{Definition of the Structure}
\label{sec:definition-structure}

Let $M$ be a (closed, Riemannian, oriented) $\Spinc$ manifold of dimension $2m-1$ with differentiable $S^1$-action. We always assume that  an $S^1$-equivariant $\Spinc$ structure $\tilde P_M$ on $M$ is chosen. In the following discussion, $\alpha$  always denotes an element of $H^1_{S^1}(M;\Z)$, which will be interpreted either as the $S^1$-invariant first cohomology group or as $S^1$-equivariant harmonic one-forms on $M$. Our quantization will be a map, which associates to a pair $(\tilde P_M, \alpha)$ an element in $K^1_{S^1}(S^1)$, being a group homomorphism in the second component.

The $\Spinc$ structure $\tilde P_M$ has a (non-splitting) associated vector bundle $\mathcal{S}_\C(M)$ that comes along with a family of Dirac operators parametrised by connection one forms on the determinant line bundle. Everything is assumed to be $S^1$-equivariant. Now fix one Dirac operator $\dirac_0$ and describe all other Dirac operators by $\dirac_\gamma := \dirac_0 + \ci c_{\gamma}$, $\gamma\in \Omega^1(M)$, where $c$ means Clifford multiplication.

Now we can associate to $\alpha$ a family $\dirac_{t\alpha}$, $t\in
[0,1]$ of Dirac operators. Since $H^1_{S^1}(M;\Z)$ is isomorphic to
the space $[M, S^1]_{S^1}$ of $S^1$-invariant homotopy classes of maps
to $S^1$, we can choose an $S^1$-invariant map $u_\alpha$ for our
fixed element $\alpha$. A direct calculation shows that
$\dirac_{\alpha} = u_\alpha^{-1} \dirac_0 u_\alpha$, which
particularly shows that the beginning and end of our familiy have the
same spectrum. Furthermore, we can interpret $u_\alpha$ as invertible
$S^1$-invariant linear map on the Hilbert space $H_M$ defined by
$\Gamma_{L^2}(\mathcal{S}_\C(M))$ (with inherited $S^1$-action). The space
$\text{Gl}_{S^1}(H_M)$ is path-connected (it is in general \emph{not} contractible, as Kuiper's theorem proves for the non-equivariant case), so we can transform $u_\alpha$ into $1$ by a family
$u_\alpha^t$. Here we have to be aware of the fact that $u_\alpha^{t}$
will in general be no bundle isomorphism but just a Hilbert space
map. Nevertheless we know that conjugating with $u_\alpha^t$ will not
change the $S^1$-equivariant spectrum of an operator on $H_M$.

So putting together the family $\dirac_{t\alpha}$ with $(u_\alpha^{t})^{1-}\dirac_\alpha u_\alpha^t$ we get a cyclic family of operators, where the spectrum is constant during the second part. Therefore, the equivariant spectral flow of this family (which equals the spectral flow of the first half) defines an element of $K^1_{S^1}(S^1)$. Now $K^1_{S^1}(S^1) \cong K^0_{S^1}(\{\text{pt}\})$; we choose an isomorphism in the following way: The embedding $\{\text{pt}\} \subset S^1$ induces an isomorphism $H_0(\{\text{pt}\}) \cong H_0(S^1)$, which by tensoring with $R(S^1)$,  Poincare-duality and the Chern character gives an isomorphism of the two spaces we talked about at the very beginning of this much too long sentence. This will be our \emph{quantization}. It is a virtual representation of $S^1$.

\begin{lemma}
  For fixed $\tilde P_M$, the quantization is a group homomorphism $Q(M)$ from $H^1_{S^1}(M;\Z)$ to $K^1_{S^1}(S^1)$. 
\end{lemma}

\begin{proof}
  Let $\dirac_0$ be a fixed Dirac operator. Then the spectral flow of $\dirac_\beta + t \ci c_{\alpha}$ is independent of $\beta$; futhermore, it is independent of the path which connects $\dirac_\beta$ and $\dirac_{\beta+ \alpha}$ by homotopy invariance of the spectral flow. Therefore, we can connect $\dirac_0$ and $\dirac_{\alpha + \beta}$ by touching $\dirac_\alpha$. This shows that the spectral flow of $\alpha + \beta$ is just the sum of the two spectral flows.
\end{proof}

\section{Going up the stairs}
\label{sec:going-up-stairs}

We now want to relate even and odd dimensional quantizations. For this section, let $M^{2m-1}$ be an odd-dimensional and $N^{2n}$ be an even-dimensional manifold; $X$ will serve as placeholder for both.

The easiest way to "go up one dimension" is to replace $X$ by $X\times S^1$. This again should be an $S^1$-manifold, so that we have to combine the action on $X$ with an action on $S^1$. Every non-trivial action on $S^1$ will force $X\times S^1$ to be a fixed-point-free space; this leads to a zero index (we already know this from the even-dimensional case, but we will later see that is also true in odd dimensions). So we take a trivial $S^1$.

Taking the only $\Spinc$ structure on $S^1$, we get a $\Spinc$ structure $\tilde P_{X\times S^1}$ on $X\times S^1$.

Let $\data(X)$ be the data we need on $X$ to define a quantization, i.e. for $\dim X$ even this is $\{\text{$\Spinc$ structures of $X$}\}$ and for $\dim X$ odd we have  $\{\text{$\Spinc$ structures of $X$}\} \times H^1_{S^1}(X;\Z)$. Now we want to find a map $\up : \data(X) \to \data(X\times S^1)$ so that the following diagram commutes:
\begin{align}
\label{eq:6}
         \xymatrix { \data(X) \ar[d]_{\up} \ar[r]^{Q(X)} & K^0_{S^1}(S^1) \\
         \data(X\times S^1) \ar[ru]_{Q(X\times S^1) }}
      \end{align}
      \begin{defi}
        Let $e_{S^1}$ be the positive generator of $H^1_{S^1}(S^1)  = H^1(S^1)$. Then we define $\up$ to be
        \begin{itemize}
        \item the map which sends $(\tilde P_M, \alpha)$ to $\tilde P_M + \alpha \cup e_{S^1}$, where addition of $\alpha \cup e_{S^1} \in H^2_{S^1}(M;\Z)$ means twisting with the respective line bundle. 
        \item or the map which sends $\tilde P_N$ to $(\tilde P_N, e_{S^1})$.
        \end{itemize}
      \end{defi}
      \begin{satz}
        The map $\up$ just defined makes \ref{eq:6} commutative.
      \end{satz}
      \begin{proof}
        We first argue for $M$:

  An equivalence of indices of this kind was mentioned in \cite{atiyah1984anomalies}. In a similar fashion, we want to construct an argument out of \cite{atiyah1976spectral} and \cite{atiyah1975spectral}. Since the periodicity of $\dirac_t$ is produced by a twist, we first look at $M \times [0,1]$ with an equivariant spectral flow from $0$ to $1$. Following \cite{atiyah1976spectral}, p.95, we can identify this spectral flow with the APS-index of $\ddt + \dirac_t$ on the manifold $M \times [0,1]$. Now the boundary terms only depend on the spectrum of $\dirac_0$ and $\dirac_1$, which is equal, so they vanish and we get the index of $\ddt + \dirac_t$ over a twisted bundle over $M \times S^1$, which is the same as the even-dimensional Dirac operator after idenfying the $\mathcal{S}^+$ and $\mathcal{S}^-$-terms with the help of $\ddt$. 

Now for $N$: If we apply $\up$ twice, we again get an even-dimensional manifold $N\times S^1\times S^1$, where the $\Spinc$-structure on $\tilde P_{N\times S^1 \times S^1}$ is given by $\tilde P_{N} + e^1_{S^1} \cup e^2_{S^1}$. If we take $\tilde P_{S^^1 \times S^1} + e^1_{S^1} \cup e^2_{S^1}$ on $S^1 \times S^1$, we get a Dirac operator $\dirac_{S^1\times S^1}$ with index 1. The construction leads to the situtation, that the index on $M\times S^1\times S^1$ ist just the product of the indices on both spaces, so we see that after applying $\up$ twice, we again have the same quantization. As we already proved the odd-dimensional case, this is enough for the even-dimensional case.
      \end{proof}

      \begin{korr}
        In every dimension, there are manifolds with nontrivial quantization.
      \end{korr}
      \begin{proof}
        We know that $Q({S^2})\neq 0$ and can then proceed by induction.
      \end{proof}

\section{The fixed point formula}
\label{sec:fixed-point-formula}

The even-dimensional $\Spinc$ quantization, described by its character $\chi$, can be calculated by a fixed point formula stemming from the Atiyah-Segal-Singer-index theorem. It looks like this:

  \begin{align*}
  \chi\big(\exp_{S^1}(v)\big) &=  \sum_{F \subset N^{S^1}} (-1)^F \cdot (-1)^{m(F)} \cdot \int_F \exp\big(\tfrac12 \tilde c_1(L|_F)\big)(v) \cdot \hat A(TF) \cdot \hat A_e(NF)(v).
  \end{align*}

Here, $\exp_{S^1}$ is the exponential map of the Lie-Group $S^1$, mapping an element $v\in \lies \cong \R$ to $S^1$, so that the left hand side is an element of $\C$. Its value in any neighbourhood of $0\in \lies$ determines $\chi$ and therefore $Q(M)$. 

On the right side, we sum integrals over the (finitely many) fixed point components. Here, $m(F)$ is the complex codimension of $F$ and $(-1)^F$ a sign depending on orientations (which is discussed in \cite{thesisfuchs} and is of no great importance here). In the integral itself we use equivariant characteristic classes for the trivial $S^1$-space $F$. Here, $\tilde c_1$ describes the equivariant first Chern class, $\hat A$ the A-roof-class and $ \hat A_e$ the quotient of the equivariant A-roof-class and Euler-class. The very last term only exists for a bundle which does not contain trivial representations and is only well-defined for small $v$. This formula was introduced in \cite{da2000quantization}; a thorough discussion and proof can be found in \cite{geomquant}. Notice that the terms in the integral are of the form 
\begin{align*}
\sum \text{(cohomology-class)} \cdot \text{(representation of $S^1$)}, 
\end{align*}
so that the integral transforms this terms into the character of a (virtual) representation (integrals of classes of the wrong dimension are defined to be zero). 

Since $N = M \times S^1$ has the same quantization as $M$, we can replace $N$ by $M \times S^1$ in this formula. Our aim is to derive a simplified version of this result, eliminating the detour over $N$ that we took.

First of all, all fixed point sets in $M \times S^1$ are of the form $F \times S^1$ because the action on $S^1$ is trivial. We especially see that all fixed point sets in $M$ are odd-dimensional; particularly, we have no isolated fixed points anymore. The integral $\int_{F\times S^1}$ will be thought of as iterated integral $\int_F \int_{S^1}$; the inner integral will be computed in the next paragraphs.

We now want to understand the three terms involved in the formula and start at the very end: The bundle $N(F \times S^1)$ in $T(M \times S^1)$ is (equivariantly) the same as the bundle $\pi_{S^1}^*(NF)$ where $\pi_{S^1} : F \times S^1 \to F$ denotes the projection (sometimes also the projection of $M \times S^1$ to $M$). Since $\pi_{S^1}^*$ is a group homomorphism commuting with characteristic classes, we can replace $A_e(N(F\times S^1))$ by $\pi_{S^1}^*(A_e(NF))$.

For $A(T(F\times S^1))$ notice that $T(F\times S^1) = \pi_{S^1}^* ( TF) \oplus \pi_{F} ^* (TS^1)$. The A-roof class is multiplicative under direct sums of vector bundles; futhermore it is just 1 on trivial bundles like $TS^1$. So we get $\pi^*_{S^1}(A(TF))$.

The first Chern class of the determinant line bundle splits into two parts: The bundle $L$ over $M$ gives us a class $\exp\big(\tfrac12 \pi^*_{S^1}(\tilde c_1(L))\big)$ while the twisting with the bundle of (equivariant) Chern class $\alpha \cup e_{S^1}$ gives an extra term $\exp(\alpha \cup e_{S^1})$. So the situation looks like this:
\begin{align*}
  \chi \big(&\exp_{S^1}(v)\big)\\
 &=  \sum_{F \subset M^{S^1}} (-1)^F \cdot (-1)^{m(F)} \cdot \int_F \int_{S^1} \exp(\alpha|_F \cup e_{S^1}) \cdot \pi_{S^1}^*\Big(\exp\big(\tfrac12 \tilde c_1(L|_F)\big)(v) \cdot \hat A(TF) \cdot \hat A_e(NF)(v)\Big) \\
  &=  \sum_{F \subset M^{S^1}} (-1)^F \cdot (-1)^{m(F)} \cdot \int_F  \bigg(\exp\big(\tfrac12 \tilde c_1(L|_F)\big)(v) \cdot \hat A(TF) \cdot \hat A_e(NF)(v) \int_{S^1} \exp(\alpha|_F \cup e_{S^1})\bigg) \\
  &=  \sum_{F \subset M^{S^1}} (-1)^F \cdot (-1)^{m(F)} \cdot \int_F  \exp\big(\tfrac12 \tilde c_1(L|_F)\big)(v) \cdot \hat A(TF) \cdot \hat A_e(NF)(v) \cdot \alpha|_F.
\end{align*}
Note that this again shows that the quantization is a group homomorphism in $H^1_{S^1}(M;\Z)$ and furthermore, that it only depends on the value of $\alpha$ on the different fixed point components (if $\alpha|_F = 0$ the component $F$ does not deliver anything for $Q(M)$).

\section{Additivity and $[Q,R]$}
\label{sec:additivity-q-r}

The important theorems for "Additivy under Cutting" and "Quantization commutes with Reduction" can now be easily transfered to the odd-dimensional world. For that we first have to make some definitions.

As in the even-dimensional case, let $Z \subset M$ be a splitting hypersurface with a free $S^1$-action. Then the reduced manifold $M_{\text{red}}$ is given by $Z/S^1$. The cut-spaces $\Mpm$ are constructed as in the even-dimensional case. To define the quantization of $M_{\text{red}}$ and $\Mpm$ we have to carry the invariant one-form $\alpha$ with us: We take $(\alpha|Z)/S^1$ on $M_{\text{red}}$ and a similar restriction on $\Mpm$. Then we have

\begin{satz}
  We have $Q(M) = Q(\Mp) + Q(\Mm)$ and $[Q,\, R]$. 
\end{satz}
\begin{proof}
  Take the manifold $M\times S^1$ with splitting hypersurface $Z \times S^1$. Since the addition of $\alpha \cup e_{S^1}$ and the constructions on the $\Spinc$ structures commute, The theorems for $M \times S^1$ imply the same ones for $M$. 
\end{proof}

\section{The situation in 3 dimensions}
\label{sec:exampl-3-dimens}

The lowest dimension that seems worth investigating is three. The fixed point set of a nontrivial $S^1$-action on $M$ consists of a finite union of circles. The fixed point formula becomes significantly easier.

\subsection{The fixed point formula revisited}
\label{sec:fixed-point-formula-1}

Since we are integrating over circles, and we already have a one-form $\alpha$ involved, we know that all integrals with further ``form-parts'' have to vanish. So it is enough to extract the ``pure representation part'' of each of the other three terms. For $\exp\big(\tfrac12 \tilde c_1(L|_F)\big)$, this is $z^{\tfrac12 \mu_F}$ (where $\mu_F$ is the degree of the representation of $S^1$ on $L|_F$). For $\hat A_e(NF)$ this is just the action of $S^1$ on the (real, two-dimensional) vector bundle $NF$ (called $z^{\tfrac12 n_F}$). So we get
\begin{align*}
  \chi(z) = \sum_F z^{\frac12 (\mu_F + n_F)}\cdot \int_F \alpha|_F.
\end{align*}

\subsection{Invariant hypersurfaces}
\label{sec:invar-hypers}

Let $Z \subset M$ be a 2-dimensional hypersurface, invariant under the action of $S^1$. The fixed point set of this action may consist of the whole of $Z$; otherwise it just consists of isolated points (or vanishes completely). Since the fixed point set on $Z$ is part of the fixed point set on $M$, which consists of odd-dimensional manifolds, the first case implies that $Z$ lies in a three-dimensional component $F$ of the fixed point set; since $F$ has to be open and closed, we have $F=M$, which is boring. 

So we rule out cases in which the action on $Z$ is trivial. This leaves us only with the cases $Z \cong S^2$ and $Z \cong T^2$, since surfaces of higher genus do not offer us non-trivial $S^1$-actions. We investigate the two cases:

\subsubsection{The 2-sphere}
\label{sec:2-sphere}

The 2-sphere lacks a fixed-point free $S^1$-action, so we cannot use it as a splitting hypersurface in the sense of quantization. Nevertheless, we can look at this surface in the context of connected sums. We equip $S^2$ with the rotational action (around the $z$-axis) of speed $l\in \Z$ (which is essentially the only action on $S^2$), calling it $S^2_l$. Now, if two manifolds $M_1$ and $M_2$ have open balls $B^3_{\pm l}$ bounding $S^2_l$ and $S^2_{-l}$ we can form an equivariant connected sum. $M_1 \# M_2$ inherits an equivariant $\Spinc$ structure, since both $\Spinc$ structures can be identified over $B^3_{\pm l}$. 

What happens to the fixed point set in this construction? The $S^1$-components not touching $S^2_{\pm l}$ do not change, the two 1-spheres going through the poles of $S^2_{\pm l}$ will be connected to one big $S^1$. Before we discuss the fixed points further, we have a look at our invariant first cohomology and line bundles:

Since the cohomology in dimension 1 and 2 is additive, we write every element of $H^1_{S^1}(M_1 \# M_2)$ as $\alpha_1 + \alpha_2$, where $\alpha_{i}$ are representatives of $H^1_{S^1}(M_i)$ with $\alpha_i|B^3_{\pm l} \equiv 0$. Every (equivariant) line bundle $L$ will be written as $L_1 \otimes L_2$, where $L_i$ is a line bundle over $M_i$ extended over $M_{1-i}$ by trivialising it over $B_{\pm l}^3$ in a non-equivariant fashion (i.e. it becomes topologically trivial with maybe nontrivial $S^1$-action). Notice that if $L_1$ and $L_2$ are the determinant line bundles for the two $\Spinc$ structures on $M_1$ and $M_2$, then the resulting $\Spinc$ structure will have $L_1 \otimes \overline{L_2}$ as determinant line bundle (in the sense just explained).

By $F_1$ and $F_2$ we denote the two fixed point components which are connected by $\#$. We have to analyse the resulting component $F$ of $M_1 \# M_2$ with the fixed point formula.

First of all, notice that the integral splits into $\int_{F_1} \alpha_1 \, + \, \int_{F_2} \alpha_2$, since we chose our one-forms to be supported only in their own ``half'' of the connected sum. The spinning number $l$ of the sphere implies that $n_F=l$. Since our determinant line bundle is given by $L_1 \otimes L_2$, we get $\mu_F = \mu_{F_1} + \mu_{F_2}$. So we have
\begin{align*}
  Q(M_1 \# M_2) &= Q(M_1) + Q(M_2) + D(F),
\end{align*}
where $D(F)$ is given as
\begin{align*}
  \big(z^{\frac12(l + \mu_{F_1} + \mu_{F_2})} - z^{\frac12(l + \mu_{F_1})}\big) \int_{F_1} \alpha_1 +  \big(z^{\frac12(l + \mu_{F_1} + \mu_{F_2})} - z^{\frac12(l + \mu_{F_2})}\big) \int_{F_2} \alpha_2.
\end{align*}

\subsubsection{The 2-torus}
\label{sec:2-torus}

Every nontrivial $S^1$-action on the two-torus is free. From the discussion above we know that toruses with free $S^1$-action are the only possible splitting hypersurfaces. The reduction then  is an $S^1$ (without a group action). We have a short look at the qunatization of such an $S^1$

There is just  one $\Spinc$ structure, combined with a one-dimensional space $H^1_{S^1}(S^1) \cong H^1(S^1) \cong \Z$. The quantization map is just the Chern character $H^1(S^1) \cong K^1(S^1)$. 

So if $M_{\text{red}}$ consists of $k$ components, we get a map of the form
\begin{align*}
  H^1(S^1) \oplus H^1(S^1) \oplus \cdots \oplus H^1(S^1) \to K^1(S^1)
\end{align*}
  where we get one summand for every component of $M_{\text{red}}$. Please note, that in the $[Q,R]$ theorem, we just apply the quantization map to the subspace of $H^1(S^1) \oplus H^1(S^1) \oplus \cdots \oplus H^1(S^1) $ which consists of the image of $\iota^*\big(H^1(M)\big)$ where $\iota : M_{\text{red}} \to M$ is the embedding.

Let us now come to some examples.



\subsection{Examples}
\label{sec:3-sphere}

For the connected sums, by get a nice playground by repeatedly connecting $S^2 \times S^1$ along ball around the poles; the action is given by a twist of speed $l$ on the $S^2$ and the trivial action on $S^1$. Of course, for $S^2 \times S^1$, the index is just given by mapping the $S^2$-index with $\up$ to $S^2 \times S^1$ (the $S^2$-index is calculated in detail in \cite{thesisfuchs} and \cite{geomquant}).

If we look at $S^3$, we can do the following: Representing elements as $(z_1, z_2)$ with $|z|=1$, we find splitting hypersurfaces for fixed $0<|z| < 1$. For every $(n_1, n_2)\in \Z^2$, we get an $S^1$-action on $S^3$. For it to be free on a surface $Z$ we have to assume that $\gcd (n_1, n_2) = 1$. If both $n_1$ and $n_2$ are non-zero, there are no fixed points on $S^3$, so $Q({S^3})$ is the zero-map. In general, we know that $R\circ Q = Q \circ R = 0$, since $H^1_{S^1}(S^3) \to H^1(S^1)$ is the zero map which implies that the quantization on the reduced manifold is always zero.


What about $T^3 = \R^3 / \Z^3$?
We get a number of $S^1$-actions on $T^3$ of the form $(z^{n_1}, z^{n_2}, z^{n_3})$. Now we choose $Z$ to consist of two 2-tori, on which the induced action is free. The splitting construction shows that $M^\pm_{\text{cut}} \cong S^3$ with two fixed point circles. The ``Additivity under Cutting'' now implies that their quantizations are the opposite of each other.

\bibliographystyle{fabian}
\bibliography{hitchin2}


\end{document}